\documentclass[12pt]{article}
\usepackage[margin=1in,footskip=0.25in]{geometry}
\usepackage{latexsym}
\usepackage{amssymb,amsmath,amsthm,accents}
\usepackage{array}
\usepackage{float}
\usepackage{color}
\usepackage{paralist}
\usepackage{graphicx}
\usepackage{caption}
\usepackage{subcaption}
\usepackage{enumitem}
\usepackage{authblk}
\usepackage{setspace}
\usepackage{etoolbox}
\AtBeginEnvironment{align*}{\let\oldfrac\frac\def\frac#1#2{\mathchoice
{\tfrac{#1}{#2}}%
{\oldfrac{#1}{#2}}%
{\oldfrac{#1}{#2}}%
{\oldfrac{#1}{#2}}}
}
\usepackage{tikz}
\newtheorem{theorem}{Theorem}[section]

\newtheorem{prop}{Proposition}

\newtheorem{con}{Conjecture}

\begin{document}
\title{\bf Subtractive Magic and Antimagic Total Labeling for Basic Families of Graphs}

\author{Inne Singgih}
\affil{University of South Carolina\\{\tt isinggih@email.sc.edu}}

\date{}
\maketitle
%%%%%%%%%%%%%%%%%%%%%%%%%%%%%%%%%%%%%%%%%%%%%%%%%%%%%%%%%%%%%%%%%%%%%%%%%%%%%%%%%%
\begin{abstract}
\noindent A \textit{subtractive arc-magic labeling} (SAML) of a directed graph $G=(V,A)$ is a bijection $\lambda :V\cup A \to \{1,2,\ldots,|V|+|A|\}$ with the property that for every $xy\in A$ we have $\lambda(xy)+\lambda(y)-\lambda(x)$ equals to an integer constant. If $\lambda(xy)+\lambda(y)-\lambda(x)$ are distinct for every $xy\in A$, then $\lambda$ is a \textit{subtractive arc-antimagic labeling} (SAAL).
A \textit{subtractive vertex-magic labeling} (SVML) of $G$ is such bijection with the property that for every $x\in V$ we have 
$\lambda(x)+\sum_{y\in V, yx \in A} \lambda(yx)-\sum_{y\in V, xy\in A} \lambda(xy)$ equals to an integer constant. If $\lambda(x)+\sum_{y\in V, yx \in A} \lambda(yx)-\sum_{y\in V, xy\in A} \lambda(xy)$ are distinct for every $x\in V$, then $\lambda$ is a \textit{subtractive vertex-antimagic labeling} (SVAL).
In this paper we prove some existence or non-existence of SAML, SVML, SAAL, and SVAL for several basic families of directed graphs, such as paths, cycles, stars, wheels, tadpoles, friendship graphs, and general butterfly graphs. The constructions are given when such labeling(s) exists. 
\end{abstract}


%%%%%%%%%%%%%%%%%%%%%%%%%%%%%%%%%%%%%%%%%%%%%%%%%%%%%%%%%%%%%%%%%%%%%%%%%%%%%%%%%%
\section{Introduction}

Let $G$ be a simple directed graph with vertex set $V$ and arc set $A$. If $xy \in A$ then $x$ is the \textit{tail} and $y$ is the \textit{head} of the arc $xy$. A \textit{total labeling} of $G$ is a bijection $\lambda :V\cup A \to \{1,2,\ldots,|V|+|A|\}$. If $x,y\in V$ and $xy\in A$, then the \textit{subtractive arc-weight} of $xy$ is defined as $wt^-(xy)=\lambda(xy)+\lambda(y)-\lambda(x)$ and the \textit{subtractive vertex-weight} of $x$ defined as $wt^-(x)=\lambda(x)+\sum_{y\in V, yx \in A} \lambda(yx)-\sum_{y\in V, xy\in A} \lambda(xy)$. A total labeling of $G$ is called SAML (or SVML) if the subtractive arc-weight of every arc (or the subtractive vertex-weight of every vertex) in $G$ are equal to an integer constant, which we call the \textit{magic constant} $\mu$. A total labeling of $G$ is called SAAL (or SVAL) if the subtractive arc-weight of every arc (or the subtractive vertex-weight of every vertex) in $G$ are all distinct. If these distinct subtractive arc-weights (or the subtractive vertex-weights) can form an arithmetic sequence starting at $a$ with difference $d$, then the total labeling is denoted by SA$(a,d)$AL (or SV$(a,d)$AL).

\noindent Barone \cite{barone} defines and proves some basic results on SAML and SVML. 

%\begin{prop} 
%\cite{barone} If an undirected graph $G$ has more than one isolated vertex, then no orientation of $G$ admits an SVML.
%\label{prop: isolated}
%\end{prop}

\begin{prop} 
\cite{barone} If $G$ is a directed cycle, $G$ admits an SAML iff $G$ admits an SVML.
\label{prop: cycle}
\end{prop}

\begin{prop} 
\cite{barone} If $\lambda$ is an SAML on $G=(V,A)$ then $\lambda'=\begin{cases} |V|+|A|+1-\lambda(x) & x\in V \\ |V|+|A|+1-\lambda(xy) & xy\in A \end{cases}$ is also an SAML. 
\label{prop: dualsaml}
\end{prop}

\begin{prop} 
\cite{barone} If $\lambda$ is an SVML on $G=(V,A)$ and $deg^+(x)-deg^-(x)=a\, \forall x\in V$, then $\lambda'=\begin{cases} |V|+|A|+1-\lambda(x) & x\in V \\ |V|+|A|+1-\lambda(xy) & xy\in A \end{cases}$ is also an SVML. 
\label{prop: dualsvml}
\end{prop}

\noindent Proposition \ref{prop: dualsaml} and \ref{prop: dualsvml} follows from the duality of magic labeling that are commonly known. Barone stated that duality is preserved in subtractive magic labeling.

\begin{prop} 
\cite{barone} Let $s$ be the length of the longest directed circuit in $G$. If $G$ has an SAML with magic constant $\mu$, then $\dfrac{s+1}{2}\leq \mu \leq \dfrac{2|V|+2|A|-s+1}{2}$.
\label{prop: mubound}
\end{prop}

\noindent Given a total labeling $\lambda$ of $G$. $\lambda$ is a \textit{strong} total labeling if $\lambda (V)=\{1,2,\ldots,|V|\}$ and is a \textit{strong$^*$} total labeling if $\lambda (A)=\{1,2,\ldots,|A|\}$. A graceful labeling of an undirected graph $G=(V,E)$ is an injective map $\phi:V\to \{1,2,\ldots,m+1\}$ such that $\{|\phi(x)-\phi(y)|:xy\in E\}=\{1,2,\ldots,m\}$. The summary of results for graceful labeling of undirected trees can be found in \cite{gallian}. 

\begin{theorem}
\cite{barone} If $T$ is a tree then $T$ has an orientation that admits a strong SAML iff $T$ has a graceful labeling.
\label{th: graceful}
\end{theorem}

\noindent In the following sections we gives some results and constructions of some subtractive magic and subtractive anti-magic labelings for paths, cycles, stars, wheels, tadpoles, friendship graphs, and butterfly graphs. For convenience, in this paper we use the notation $\{x_i\}_{i=1}^m$ to denotes the set $\{x_1,x_2,\ldots,x_m\}$. When labeling exists, only one construction is given and the duals are left for the reader to explore.

%%%%%%%%%%%%%%%%%%%%%%%%%%%%%%%%%%%%%%%%%%%%%%%%%%%%%%%%%%%%%%%%%%%%%%%%%%%%%%%%%%
\section{Paths}
In this paper the notation $P_n$ is used for a directed path (dipath) with $n$ vertices.

\begin{theorem}
Dipath $P_n$ has a strong SAML for any $n$ with magic constant $\mu=n$.
\end{theorem}
\begin{proof}
In \cite{4end} Rosa et al. proved that undirected trees that has at most 4 end-vertices has a graceful labeling. By Theorem \ref{th: graceful}, dipath $P_n$ admits a strong SAML. The construction of graceful labeling for undirected paths is given in multiple results, which summary can be found at \cite{gallian}. The conversion from graceful labeling to strong SAML is given in the proof of Theorem \ref{th: graceful}. The construction of the strong SAML can be written directly as follows:\\
Let $V=\{v_i\}_{i=1}^n$ and $A=\{a_i\}_{i=1}^{n-1}$ where $a_i=v_{i+1}v_i$ if $i$ is odd and $a_i=v_iv_{i+1}$ if $i$ is even.\\
Let $\lambda(v_i)=\frac{i+1}{2}$ when $i$ is odd, $\lambda(v_i)=n+1-\frac{i}{2}$ when $i$ is even, and $\lambda(a_i)=2n-i$. Then
$$wt^-(a_i)=\begin{cases} 
\lambda(a_i)+\lambda(v_i)-\lambda(v_{i+1})=(2n-i)+\frac{i+1}{2}-\left(n+1-\frac{i+1}{2}\right)=n & \text{if }i\text{ is odd}\\
\lambda(a_i)+\lambda(v_{i+1})-\lambda(v_i)=(2n-i)+\frac{(i+1)+1}{2}-\left(n+1-\frac{i}{2}\right)=n & \text{if }i\text{ is even}
\end{cases} $$
Barone in \cite{barone} describes the orientation as: for any arc, the vertex with smaller label is the head and the one with larger label is the tail.
\end{proof}

\vspace{2mm}
\begin{theorem}
Dipath $P_n$ does not have SVML for any $n$.
\end{theorem}
\begin{proof}
Suppose there is an SVML $\lambda$ for $P_n$ with magic constant $\mu$. It is easy to see that $\mu \not\in \{1,2,\ldots,2n-1\}$. In any orientation of $P_n$ there must be a vertex that only serves as a tail for all arc(s) incident to it. Let $x$ be such vertex with largest label, then must have $\lambda(x)=2n-1$. Comparing $\mu$ with $\lambda(x)$ we have that $\mu\leq 2n-2$ if $x$ is an endpoint of $P_n$, and $\mu\leq 2n-4$ if $x$ is not an endpoint of $P_n$. Contradiction. 
\end{proof}

\vspace{2mm}
\begin{theorem}
Dipath $P_n$ has a strong SA$(n+3,1)$AL for all $n>1$.
\end{theorem}
\begin{proof}
Let $V=\{v_i\}_{i=1}^n$ and $A=\{a_i\}_{i=1}^{n-1}$ where $a_i=v_iv_{i+1}$ for $1\leq i\leq n-1$.\\
Let $\lambda(v_i)=i$ for $1\leq i \leq n$ and $\lambda(a_i)=2n+1-i$ for $1\leq i \leq n-1$, then:
$$wt^-(a_i)=\lambda(a_i)+\lambda(v_{i+1})-\lambda(v_i)=2n+2-i \quad \text{for } 1\leq i \leq n-1 $$
Hence $wt^-(a_i)\in \{n+3, n+4, \ldots, 2n+1\}$ for $1\leq i \leq n-1$.
\end{proof}

\vspace{2mm}
\begin{theorem}
Dipath $P_n$ has a strong$^*$ SV$(n,1)$AL for all $n>1$.
\end{theorem}
\begin{proof}
Let $V=\{v_i\}_{i=1}^n$ and $A=\{a_i\}_{i=1}^{n-1}$ where $a_i=v_iv_{i+1}$ for $1\leq i\leq n-1$.\\
Let $\lambda(v_i)=2n-i$ for $1\leq i \leq n$ and $\lambda(a_i)=i$ for $1\leq i \leq n-1$, then:
\begin{align*}
wt^-(v_1) &= \lambda(v_1)-\lambda(a_1)=2n-2\\
wt^-(v_n) &= \lambda(v_n)+\lambda(a_{n-1})=2n-1\\
wt^-(v_i) &= \lambda(v_i)+\lambda(a_{i-1})-\lambda(a_i)=2n-1-i \quad \text{for } 2\leq i \leq n-1
\end{align*} 
Hence $wt^-(v_i)\in \{n, n+1, \ldots, 2n-1\}$ for $1\leq i \leq n$.
\end{proof}

%%%%%%%%%%%%%%%%%%%%%%%%%%%%%%%%%%%%%%%%%%%%%%%%%%%%%%%%%%%%%%%%%%%%%%%%%%%%%%%%%%
\section{Cycles}

\begin{theorem}
Directed cycle (dicycle) $C_n$ have neither SAML nor SVML for any $n$.
\end{theorem}
\begin{proof}
Suppose there is an SAML $\lambda$ for $C_n$ with magic constant $\mu$. It is easy to see that $\mu \not\in \{1,2,\ldots,2n\}$. Using $s=n$ for Proposition \ref{prop: mubound} we have $\frac{n+1}{2}\leq \mu \leq \frac{3n+1}{2}$. Contradiction. By Proposition \ref{prop: cycle} $C_n$ also does not have SVML.
\end{proof}

\vspace{2mm}
\begin{theorem}
Dicycle $C_n$ has a strong SA$(n+1,1)$AL and a strong SV$(1,1)$AL.
\end{theorem}
\begin{proof}
Let $V=\{v_i\}_{i=1}^n$ and $A=\{a_i\}_{i=1}^n$ where $a_i=v_iv_{i+1}$ for $1\leq i\leq n-1$ and $a_n=v_nv_1$. Let $\lambda(v_i)=i$ for $1\leq i \leq n$, $\lambda(a_i)=2n-i$ for $1\leq 1 \leq n-1$, and $\lambda(a_n)=2n$. Then:
\begin{align*}
wt^-(a_n) &=\lambda(a_n)+\lambda(v_1)-\lambda(v_n)=n+1\\
wt^-(a_i) &=\lambda(a_i)+\lambda(v_{i+1})-\lambda(v_i)=2n+1-i \quad \text{for } 1\leq i \leq n-1
\end{align*}
Hence $wt^-(a_i)\in \{n+1, n+2, \ldots, 2n\}$ for $1\leq i \leq n$. Also:
\begin{align*}
wt^-(v_n) &= \lambda(v_n)+\lambda(a_{n-1})-\lambda(a_n)=1\\
wt^-(v_i) &= \lambda(v_i)+\lambda(a_{i-1})-\lambda(a_i)=i+1 \quad \text{for } 1\leq i \leq n-1
\end{align*}
Hence $wt^-(v_i)\in \{1, 2, \ldots, n\}$ for $1\leq i \leq n$.
\end{proof}

%%%%%%%%%%%%%%%%%%%%%%%%%%%%%%%%%%%%%%%%%%%%%%%%%%%%%%%%%%%%%%%%%%%%%%%%%%%%%%%%%%
\section{Stars}
Let $S_n$ denotes a directed star (distar) with $v_0$ as the center and $v_1,v_2,\ldots, v_n$ as the leaves.

\begin{theorem}
Distar $S_n$ has a strong SAML with magic constant $\mu=2(n+1)$ for all $n\geq 1$.
\end{theorem}
\begin{proof}
Let $\lambda(v_0)=1$ and $\lambda(v_i)=i+1$ for $1\leq i \leq n$.\\
Let $a_i$ denotes the arc $v_0v_i$ and let $\lambda(a_i)=2(n+1)-i$ for $1\leq i \leq n$.\\
Then $wt^-(a_i)=\lambda(a_i)+\lambda(v_i)-\lambda(v_0)=2(n+1)$.
\end{proof}

\vspace{2mm}
\begin{theorem}
Distar $S_n$ has an SA$(2n+2,2)$AL for all $n$.
\end{theorem}
\begin{proof}
Let $\lambda(v_0)=2n+1$ and $\lambda(v_i)=i$ for $1\leq i \leq n$.\\
Let $a_i$ denotes the arc $v_iv_0$ and let $\lambda(a_i)=2n+1-i$ for $1\leq i \leq n$.\\
Then $wt^-(a_i)=\lambda(a_i)+\lambda(v_0)-\lambda(v_i)=4n+2-2i$ for $1\leq i \leq n$.\\
So we have $wt^-(a_i)\in \{2n+2,2n+4,\ldots,4n\}$.
\end{proof}

\vspace{2mm}
\begin{theorem}
Distar $S_n$ has a strong$^*$ SVAL for all $n$.
\end{theorem}
\begin{proof}
Let $\lambda(v_0)=1$ and $\lambda(v_i)=n+1+i$ for $1\leq i \leq n$.\\
Let $a_i$ denotes the arc $v_iv_0$ and let $\lambda(a_i)=n+2-i$ for $1\leq i \leq n$. Then
\begin{align*}
wt^-(v_0)&=\lambda(v_0)+\sum_i \lambda(a_i)=1+2+\ldots+(n+1)=\frac{1}{2}(n+1)(n+2)\\
wt^-(v_i)&=\lambda(v_i)-\lambda(a_i)=2i-1 \quad \text{for }1\leq i \leq n
\end{align*}
Hence $wt^-(v_i)\in \{1,3,5,\ldots,2n-1\}$ for $1\leq i \leq n$.\\
Since $\frac{1}{2}(n+1)(n+2)>2n-1$ for all $n$, we have all $wt^-$ are distinct.
\end{proof}

\vspace{2mm}
\begin{con}
Distar $S_n$ does not have an SVML for any $n$.
\end{con}
\vspace{-4mm}
\noindent Suppose there is an SVML for $S_n$ with magic constant $\mu$. One can verify that the center $v_0$ can not be exclusive tail or exclusive head. Since $v_0$ serves as head for some arc(s), then the label of the tail of those arc(s) forces $\mu<2n+1$, that is, $\mu$ is one of the label. Hence $\mu$ can only be either the label of an arc $v_iv_0$ ($v_0$ as head), or the label of $v_0$. If wlog $\mu=\lambda(v_1v_0)$, then $\lambda(v_1)=2\mu$ and so $\mu \leq n$. Taking the other $n-1$ pairs $(\lambda(v_i),\lambda(a_i))$ from $\{1,2,\ldots,2n+1\}-\{\mu,2\mu\}$ for the star's legs, we have the leftover label for $v_0$ is in $\{n+1,n+2,\ldots,2n+1\}-\{\mu\}$. But then $wt^-(v_0)>n$ since there are more arcs with $v_0$ as head, and the arcs with $v_0$ as tails has smaller labels, and so $wt^-(v_0)>\mu$. Contradiction. The only case left is when $\lambda(v_0)=\mu$.\\
%%%%%%%%%%%%%%%%%%%%%%%%%%%%%%%%%%%%%%%%%%%%%%%%%%%%%%%%%%%%%%%%%%%%%%%%%%%%%%%%%%
\section{Wheels}

Let $W_n$ denotes a directed wheel with $V(W_n)=\{v_0,v_1,\ldots,v_n\}$, where $v_0$ is the center. 

\begin{theorem}
Directed wheel $W_n$ has an SVAL labeling for $n\geq 3$. For $n=3$, it is an SV$(n+1,2)$AL.
\end{theorem}
\begin{proof}
Let $A(W_n)=\{a_1,a_2,\ldots,a_n,b_1,b_2,\ldots,b_n\}$ where $a_i=v_iv_0$ for $1\leq i \leq n$, $b_i=v_iv_{i+1}$ for $1\leq i \leq n-1$, and $b_n=v_nv_1$.\\
Let $\lambda(v_0)=1$ and $\lambda(v_i)=3n+1-i$ for $1\leq i \leq n-1$, and $\lambda(v_n)=3n+1$.\\
Let $\lambda(a_i)=i+1$ for $1\leq i \leq n$, $\lambda(b_i)=n+2+i$ for $1\leq i \leq n-1$, and $\lambda(b_n)=n+2$. Then
\begin{align*}
wt^-(v_0) &= \lambda(v_0)+\sum_{i} \lambda(a_i)=\frac{1}{2}(n+1)(n+2)\\
wt^-(v_i) &= \lambda(v_i)+\lambda(b_{i-1})-\lambda(b_i)-\lambda(a_i)=3n+1-2i \quad \text{for } 1\leq i \leq n-1\\
wt^-(v_n) &= \lambda(v_n)+\lambda(b_{n-1})-\lambda(b_n)-\lambda(a_n)=3n-1
\end{align*}
Hence $\{wt^-(v_i)\}_{i=1}^n=\{n+1,n+3,\ldots,3n-1\}$. Since $\frac{1}{2}(n+1)(n+2)>3n-1$ for all $n\geq 3$, we have all $wt^-$ are distinct. For $n=3$, $\frac{1}{2}(n+1)(n+2)=10$ while $3n-1=8$, so we have an SV$(n+1,2)$AL. For larger $n$, $wt^-(v_0)=\frac{1}{2}(n+1)(n+2)$ grows faster as $n$ increases.
\end{proof}

%%%%%%%%%%%%%%%%%%%%%%%%%%%%%%%%%%%%%%%%%%%%%%%%%%%%%%%%%%%%%%%%%%%%%%%%%%%%%%%%%%
\section{Tadpoles}

Let $(n,t)$-tadpole denotes a directed tadpole obtained by connecting a dicycle $C_n$ and a dipath $P_t$ using an arc. Let the vertex set of the cycle part be $\{v_i\}_{i=1}^n$ and the vertex set of the path part be $\{u_i\}_{i=1}^t$. Let the arc set be $\{a_i\}_{i=1}^n \cup \{b_i\}_{i=1}^{t-1} \cup \{c\}$ where $a_i=v_iv_{i+1}$ for $1\leq i\leq n-1$, $a_n=v_nv_1$, $b_i=u_iu_{i+1}$, and $c=u_tv_1$.\\

\begin{theorem}
$(n,t)$-tadpole has a strong SAAL.
\end{theorem}
\begin{proof}
Let $\lambda(v_1)=t+1$, $\lambda(v_i)=n+t+2-i$ for $2\leq i \leq n$, and $\lambda(u_i)=i$ for $1 \leq i \leq t$.\\
Let $\lambda(a_i)=n+t+i$ for $1\leq i \leq n$, $\lambda(b_i)=2n+2t+2-i$ for $1\leq i \leq t$, and $\lambda(c)=2n+t+1$. Then
\begin{align*}
wt^-(a_1) &= \lambda(a_1)+\lambda(v_2)-\lambda(v_1)=2n+t\\
wt^-(a_i) &= \lambda(a_i)+\lambda(v_{i+1})-\lambda(v_i)=n+t-1+i \quad \text{for }2\leq i\leq n\\
wt^-(c) &= \lambda(c)+\lambda(v_1)-\lambda(u_t)=2n+t+2\\
wt^-(b_i) &= \lambda(b_i)+\lambda(u_i)-\lambda(u_{i-1})=2n+2t+3-i
\end{align*}
Hence we have $\{wt^-(a_i)\}_{i=1}^n=\{n+t+1,n+t+2,\ldots,2n+t\}$ and $\{wt^-(c)\}\cup \{wt^-(b_i)\}_{i=1}^t=\{2n+t+2,2n+t+3,\ldots,2n+2t+2\}$. Since the weight $2n+t+1$ does not exist then $\lambda$ is an SAAL but not an SA$(a,d)$AL.
\end{proof}

\vspace{2mm}
\begin{theorem}
$(n,t)$-tadpole has a strong$^*$ SV$(n+t+1,1)$AL.
\end{theorem}
\begin{proof}
Let $\lambda(v_1)=n+t+1$, $\lambda(v_i)=2n+t+2-i$ for $2\leq i \leq n$, and $\lambda(u_i)=2n+2t+1-i$ for $1 \leq i \leq t$. 
Let $\lambda(a_i)=t+i$ for $1\leq i \leq n$, $\lambda(b_i)=i$ for $1\leq i \leq t$, and $\lambda(c)=t$. Then
\begin{align*}
wt^-(v_1) &= \lambda(v_1)+\lambda(a_n)+\lambda(c)-\lambda(a_1)=2n+2t\\
wt^-(v_i) &= \lambda(v_i)+\lambda(a_{i-1})-\lambda(a_i)=2n+t+1-i \quad \text{for }2\leq i\leq n\\
wt^-(u_1) &= \lambda(u_1)-\lambda(b_1)=2n+2t-1\\
wt^-(u_i) &= \lambda(u_i)+\lambda(b_{i-1})-\lambda(b_i)=2n+2t-i \quad \text{for }2\leq i\leq t-1\\
wt^-(u_t) &= \lambda(u_t)+\lambda(b_{t-1})-\lambda(c)=2n+t
\end{align*}
Hence we have the subtractive vertex-weights are $\{n+t+1, n+t+2,\ldots,2n+2t\}$.
\end{proof}

%%%%%%%%%%%%%%%%%%%%%%%%%%%%%%%%%%%%%%%%%%%%%%%%%%%%%%%%%%%%%%%%%%%%%%%%%%%%%%%%%%
\section{Friendship and General Butterfly}
Let $F_n$ denotes a directed friendship graph constructed by joining $n$ copies of the dicycle $C_3$ with a common vertex, call this common vertex $x$. Let $\{v_{i1},v_{i2}\}_{i=1}^n$ be the other two vertices in the $i\textsuperscript{th}$ dicycle. For $1\leq i \leq n$ let the arcs be $a_{i0}=xv_{i1}$, $a_{i1}=v_{i1}v_{i2}$, and $a_{i2}=v_{i2}x$.\\

\noindent For convenience define the ordered $m$-tuple notation: 
\begin{align*}
\lambda(p_i,p_2,\ldots,p_m)=(q_1,q_2,\ldots,q_m) &\text{ means } \lambda(p_j)=q_j \text{ for } 1\leq j\leq m\\
wt^-(p_i,p_2,\ldots,p_m)=(q_1,q_2,\ldots,q_m) &\text{ means } wt^-(p_j)=q_j \text{ for } 1\leq j\leq m
\end{align*}

\vspace{2mm}
\begin{theorem}
Directed friendship graph $F_n$ has a strong SA$(2n+2,1)$AL. 
\end{theorem}
\begin{proof}
Let $\lambda(x)=1$, and let
\begin{align*}
\lambda(v_{11},v_{21},\ldots,v_{n1}) &= (2,3,\ldots,n+1) \\
\lambda(v_{12},v_{22},\ldots,v_{n2}) &= (n+2, n+3,\ldots, 2n+1) \\
\lambda(a_{10},a_{20},\ldots,a_{n0}) &= (2n+2,2n+3,\ldots, 3n+1) \\
\lambda(a_{11},a_{21},\ldots,a_{n1}) &= (3n+2,3n+3,\ldots, 4n+1) \\
\lambda(a_{n2},a_{(n-1)2},\ldots, a_{12}) &= (4n+2, 4n+3, \ldots, 5n+1)
\end{align*}
Hence we have
\begin{align*}
wt^-(a_{n2},a_{(n-1)2},\ldots, a_{12}) &= (2n+2,2n+4,\ldots, 4n)\\
wt^-(a_{10},a_{20},\ldots,a_{n0}) &= (2n+3,2n+5,\ldots,4n+1) \\
wt^-(a_{11},a_{21},\ldots,a_{n1}) &= (4n+2,4n+3,\ldots,5n+1)
\end{align*}
Combined we have the subtractive arc-weigts are $\{2n+2,2n+3,\ldots,5n+1\}$.
\end{proof}

\vspace{5mm}
\noindent A butterfly graph (bow-tie graph, hourglass graph) is a planar graph constructed by joining two copies of the cycle graph $C_3$ with a common vertex and is therefore isomorphic to the friendship graph $F_2$. Let $E_n$ denotes the directed general butterfly graph that constructed by joining two copies of the dicycle $C_n$ at a common vertex. Note that $E_3\cong F_2$.\\
Let $\{v_i\}_{i=1}^n$ and $\{u_i\}_{i=1}^n$ be the vertices of the two dicycles and let $x=v_n=u_n$.\\
Let $a_i=v_iv_{i+1}$ and $b_i=u_iu_{i+1}$ for $1\leq i\leq n-1$, and let $a_n=v_nv_1$ and $b_n=u_nu_1$.\\

\begin{theorem}
Directed general butterfly graph $E_n$ has a strong SA$(2n,1)$AL.
\end{theorem}
\begin{proof}
Let $\lambda(x)=2n-1$. Let $\lambda(v_i)=2n-1-2i$ and $\lambda(u_i)=2n-2i$ for $1\leq i \leq n-1$.\\
Let $\lambda(a_i)=4n-1-2i$ and $\lambda(b_i)=4n-2-2i$ for $1\leq i \leq n-2$.\\
Let $\lambda(a_{n-1},b_{n-1},a_n,b_n)=(2n+1,2n,4n-2,4n-1)$. Then
\begin{align*}
wt^-(a_i) &= \lambda(a_i)+\lambda(v_{i+1})-\lambda(v_i)=4n-3-2i \quad \text{for } 1\leq i \leq n-2\\
wt^-(b_i) &= \lambda(b_i)+\lambda(u_{i+1})-\lambda(u_i)=4n-4-2i \quad \text{for } 1\leq i \leq n-2\\
wt^-(a_{n-1}) &= \lambda(a_{n-1})+\lambda(x)-\lambda(v_{n-1})=4n-1\\
wt^-(b_{n-1}) &= \lambda(b_{n-1})+\lambda(x)-\lambda(u_{n-1})=4n-3\\
wt^-(a_n) &= \lambda(a_n)+\lambda(v_1)-\lambda(x)=4n-4\\
wt^-(b_n) &= \lambda(b_n)+\lambda(u_1)-\lambda(x)=4n-2
\end{align*}
Combined we have the subtractive arc-weigts are $\{2n,2n+1,\ldots,4n-1\}$.
\end{proof}

\vspace{2mm}
\begin{theorem}
Directed general butterfly graph $E_n$ has a strong$^*$ SVAL.
\end{theorem}
\begin{proof}
Let $\lambda(x)=4n-1$. Let $\lambda(v_i)=2n-1+2i$ and $\lambda(u_i)=2n+2i$ for $1\leq i \leq n-1$.\\
Let $\lambda(a_i)=2n-1-2i$ and $\lambda(b_i)=2n-2i$ for $1\leq i \leq n-1$.\\
Let $\lambda(a_n)=2n-1$ and $\lambda(b_n)=2n$. Then
\begin{align*}
wt^-(x) &= \lambda(x)+\lambda(a_{n-1})+\lambda(b_{n-1})-\lambda(a_n)-\lambda(b_n)=3\\
wt^-(v_1) &= \lambda(v_1)+\lambda(a_n)-\lambda(a_1)=2n+3\\
wt^-(u_1) &= \lambda(u_1)+\lambda(b_n)-\lambda(b_1)=2n+4\\
wt^-(v_i) &= \lambda(v_i)+\lambda(a_{i-1})-\lambda(a_i)=2n+1+2i \quad \text{for } 2\leq i \leq n-1\\
wt^-(u_i) &= \lambda(u_i)+\lambda(b_{i-1})-\lambda(b_i)=2n+2+2i \quad \text{for } 2\leq i \leq n-1
\end{align*}
Combined we have the subtractive arc-weigts are $\{3\} \cup \{2n+3,2n+4,\ldots,4n\}$.\\ 
Since $2n+3>3$ for all $n$ we have all $wt^-$ are distinct.
\end{proof}

%%%%%%%%%%%%%%%%%%%%%%%%%%%%%%%%%%%%%%%%%%%%%%%%%%%%%%%%%%%%%%%%%%%%%%%%%%%%%%%%%%
%%%%%%%%%%%%%%%%%%%%%%%%%%%%%%%%%%%%%%%%%%%%%%%%%%%%%%%%%%%%%%%%%%%%%%%%%%%%%%%%%%
%\pagebreak

\end{document}